\documentclass[12pt]{amsart}

\usepackage{amssymb,latexsym,amsmath,mathtools,nicefrac,bbm}
\usepackage{stmaryrd}
\usepackage{graphicx}
\usepackage{psfrag}
\usepackage{color}
\usepackage{enumitem}
\usepackage{comment}
\usepackage{bm}
\setlist[enumerate]{parsep=0pt plus 4pt,topsep=0pt plus 4pt}
\usepackage{url}

\usepackage{hyperref}
\allowdisplaybreaks

\newcommand\excise[1]{}
\renewcommand\comment[1]{{$\star$\sf\textbf{#1}$\star$}}
\newtheorem{thm}{Theorem}[section]
\newtheorem{lemma}[thm]{Lemma}

\newtheorem{cor}[thm]{Corollary}
\newtheorem{prop}[thm]{Proposition}

\theoremstyle{definition}
\newtheorem{example}[thm]{Example}
\newtheorem{remark}[thm]{Remark}
\newtheorem{defn}[thm]{Definition}

\numberwithin{equation}{section}

\newcounter{separated-sec}
\newcounter{separated}
\newcounter{sylvan-sec}

\newcommand{\Ring}[1]{\ensuremath{\mathbb{#1}}}

\renewcommand\>{\rangle}
\newcommand\<{\langle}
\newcommand\0{\mathbf{0}}
\newcommand\1{\mathbbm{1}}
\newcommand\CC{{\widetilde C}{}{}}
\newcommand\HH{{\widetilde H}{}{}}

\newcommand\NN{\Ring{N}}

\newcommand\ZZ{\Ring{Z}}
\newcommand\bb{{\mathbf b}}
\newcommand\cc{{\mathbf c}}

\newcommand\kk{\Bbbk}

\newcommand\xx{{\mathbf x}}

\newcommand\del{\partial}
\renewcommand\aa{{\mathbf a}}
\renewcommand\phi{\varphi}
\newcommand\uu{\mathbf{u}}

\newcommand\vv{\mathbf{v}}
\newcommand\aalpha{{\bm{\alpha}}}
\newcommand\bbeta{{\bm{\beta}}}

\newcommand\from{\leftarrow}
\newcommand\into{\hookrightarrow}

\newcommand\onto{\twoheadrightarrow}

\newcommand\spot{{\hbox{\raisebox{1pt}{\tiny$\scriptscriptstyle{\bullet}$}}}}

\newcommand\nothing{\varnothing}

\newcommand\filleftmap{\mathord\leftarrow \mkern-6mu
	\cleaders\hbox{$\mkern-2mu \mathord- \mkern-2mu$}\hfill
	\mkern-6mu \mathord-}

\renewcommand\epsilon{\varepsilon}

\newcommand\st{\mathit{ST\hspace{-.2ex}}}

\newcommand\wt[1]{{\widetilde{#1}}}
\newcommand\mkl[1]{\makebox[0pt][l]{$#1$}}

\DeclareMathOperator\tor{Tor} 

\DeclareMathOperator\image{im} 

\newcommand\monomialmatrix[3]{{
\begin{array}{@{}r@{\:}r@{}c@{}l@{}}
  \begin{array}{@{}c@{}}		
	\begin{array}{@{}r@{}}
	\\
	#1
	\end{array}\!
  \end{array}						
&
  \begin{array}{@{}c@{}}		
	\begin{array}{@{}l@{}}\\				
	\end{array}						
	\\							
	\left[\begin{array}{@{}l@{}}				
	#3							
	\end{array}\!						
	\right.							
  \end{array}							
&
  #2					
&
  \begin{array}{@{}c@{}}		
	\begin{array}{@{}l@{}}\\				
	\end{array}						
	\\							
	\left.\!\begin{array}{@{}l@{}}				
	#3							
	\end{array}						
	\right]							
  \end{array}							
\end{array}
}}

\newcommand\red[1]{\color{red}#1}
\newcommand\blu[1]{\color{blue}#1}

\begin{document}

\mbox{}
\title{Minimal resolutions of lattice ideals}
\author{Yupeng Li}
\address{Mathematics Department\\Duke University\\Durham, NC 27708}
\urladdr{\url{http://sites.duke.edu/ypli}}
\author{Ezra Miller}
\address{Mathematics Department\\Duke University\\Durham, NC 27708}
\urladdr{\url{http://math.duke.edu/people/ezra-miller}}
\author{Erika Ordog}
\address{Mathematics Department\\Texas A$\&$M University\\College Station, TX 77843}
\urladdr{\url{http://www.math.tamu.edu/~erika.ordog}}

\makeatletter
  \@namedef{subjclassname@2020}{\textup{2020} Mathematics Subject Classification}
\makeatother
\subjclass[2020]{Primary: 13F65, 13D02, 05E40, 13F20, 13C13, 20M25;
Secondary: 05E45, 68W30, 20M14}

\date{16 September 2024}

\begin{abstract}
A canonical minimal free resolution of an arbitrary co-artinian
lattice ideal over the polynomial ring is constructed over any field
whose characteristic is $0$ or any but finitely many positive primes.
The differential has a closed-form combinatorial description as a sum
over lattice paths in $\ZZ^n$ of weights that come from sequences of
faces in simplicial complexes indexed by lattice points.  Over a field
of any characteristic, a non-canonical but simpler resolution is
constructed by selecting choices of higher-dimensional analogues of
spanning trees along lattice paths.  These constructions generalize
sylvan resolutions for monomial ideals by lifting them equivariantly
to lattice modules.
\end{abstract}
\maketitle

\tableofcontents

\section{Introduction}\label{s:intro}

One of the goals of commutative algebra is to construct free
resolutions of ideals over the polynomial ring $R = \kk
[x_1,\dots,x_n] = \kk[\xx]$ in $n$ variables over a field~$\kk$.
Combinatorial settings, such as that of determinantal or toric ideals,
provide models for the general theory.  For the class of lattice
ideals, of which the toric ideals are the prime examples,
the input data is a lattice~$L$ (that is, a subgroup) in $\ZZ^n$ whose
intersection with the nonnegative orthant is trivial: $L \cap \NN^n =
\{\0\}$; see \cite[Chapter~7]{cca}.  The \textit{lattice ideal} $I_L =
\<\xx^\uu - \xx^\vv \mid \uu, \vv \in \NN^n \text{ and } \uu - \vv \in
L\>$ is homogeneous with respect to a grading in which the degree of
each variable is a positive integer because
$L \cap \NN^n = \{\0\}$.

Many beautiful constructions of free resolutions of lattice ideals are
known in various settings.  For example, Peeva and Sturmfels
\cite{ps98a} present a minimal free resolution built from
combinatorial quadrangle resolutions given any codimension~$2$ lattice
ideal.  The same authors \cite{ps98b} construct the Scarf complex,
which yields a canonical combinatorial resolution of any lattice ideal
over a field of any characteristic; that resolution is minimal when
the ideal is generic.  One can also associate a lattice ideal $I_L$
with a graph~$G$ via the image~$L$ of the graph Laplacian matrix.
For complete undirected graphs, Manjunath and Sturmfels
\cite{manjunath-sturmfels2013} study the minimal free resolution of
$I_L$ by constructing the complex $\mathcal{CYC}_G$, which coincides
with the Scarf complex.  Subsequently, for connected undirected
graphs, Manjunath--Schreyer--Wilmes
\cite{manjunath-schreyer-wilmes2015} and Mohammadi--Shokrieh
\cite{mohammadi-shokrieh2014}, specify a minimal free resolution
of~$I_L$ using combinatorial graph theory and Gr\"obner bases.  More
recently, for directed graphs, O'Carroll and Planas-Vilanova
\cite{ocarroll--planas-vilanova2018} construct a free resolution
of~$I_L$ as the chain complex $\mathcal{CYC}_G$ associated to a
finite, strongly connected, weighted digraph.  The resolution is
minimal if and only if the digraph is strongly complete.  Other
properties and combinatorial descriptions of various resolutions of
lattice ideals can be found in
\cite{briales-morales--pison-casares--vigneron-tenorio2001,
charalambous-thoma2010, pison2003} and references therein.

Along the lines of canonical but perhaps not minimal resolutions,
Bayer and Sturmfels \cite{bayer-sturmfels1998} (see also
\cite[Chapter~9]{cca}) generalize the Scarf construction to the hull
resolution: a canonical cellular resolution for any monomial module.
The monomials with exponent vectors in a lattice~$L$ generate the
\textit{lattice module} $M_L = R \{\xx^\aa \mid\nolinebreak \aa
\in\nolinebreak L\}$.  When the lattice ideal $I_L$ is generic, the
resolution of~$M_L$ is minimal.  Any free resolution of the lattice
module $M_L$ over $R[L]$ descends functorially to a resolution of the
lattice ideal $I_L$ over~$R$ in a way that preserves minimality
\cite[Corollary~3.3]{bayer-sturmfels1998}.

Although the module structures of minimal resolutions of lattice
ideals were explicitly described by Briales-Morales, Pis\'on-Casares,
and Vigneron-Tenorio
\cite{briales-morales--pison-casares--vigneron-tenorio2001}, with
differentials filled in algorithmically, none of the known closed-form
combinatorial constructions for arbitrary lattice ideals are minimal.
The best result along these lines is by Tchernev \cite{tchernev2019}:
an explicit recursive algorithm for canonical minimal resolutions of
toric rings, where the lattice ideal is prime, using dynamical systems
on chain complexes.  His method works over a field of any
characteristic by using a transcendental extension of the base field
when necessary.

Until now there has been no closed-form description of the
differentials in any family of minimal resolutions that encompasses
all toric ideals, let alone all lattice ideals.  The main result of
this paper is the construction of a canonical minimal free resolution
of an arbitrary positively graded lattice ideal with a closed-form
combinatorial description of the differential in characteristic~$0$
and all but finitely many positive characteristics
(Theorem~\ref{t:sylvan}.\ref{i:thm-canonical} and
Remark~\ref{r:sylvan}.\ref{i:thm-canonical}).  It generalizes the
sylvan resolution for monomial ideals \cite[Theorem~3.7]{sylvan},
which is the first canonical closed-form combinatorial minimal free
resolution for arbitrary monomial ideals; it works in
characteristic~$0$ and all but finitely many positive characteristics.
Our resolutions of lattice ideals first use the sylvan construction to
minimally resolve the lattice module~$M_L$ over the polynomial
ring~$R$ in a canonical way, so that the resolution is equipped with a
natural free action of the lattice~$L$.  Thus the $R$-resolution of
$M_L$ is a canonical minimal resolution of~$M_L$ as an $R[L]$ module,
so it descends via the Bayer--Sturmfels functor to a resolution
of~$I_L$ over~$R$ by \cite[Corollary~3.3]{bayer-sturmfels1998}.

We provide a similar construction over a field of any characteristic
(Theorem~\ref{t:sylvan}.\ref{i:thm-noncanonical} and
Remark~\ref{r:sylvan}.\ref{i:thm-noncanonical}), but the resolution in
this case is non-canonical, since it involves choices of
higher-dimensional analogues of spanning trees
(Definition~\ref{d:hedge}.\ref{i:community}).

\subsubsection*{Acknowledgments}

\section{Overview of the construction}\label{s:overview}

A lattice module is a type of \textit{monomial module}, an
$R$-submodule of the Laurent polynomial ring $\kk [\xx_1^{\pm 1},
\ldots, \xx_n^{\pm 1}]$ generated by monomials $\xx^\aa$ for vectors
$\aa \in \ZZ^n$.  In addition to ensuring a positive grading, the
condition $L \cap \NN^n = \{\0\}$ means that the lattice module~$M_L$
is \textit{co-artinian}, which for a general monomial module~$M$ means
that it is generated by its set of minimal monomials:
$$
  \text{min}(M)
  =
  \{\xx^\aa \in M \mid \xx^\aa/x_i \notin M \text{ for all } i\};
$$
equivalently, the set of monomials in~$M$ with degree $\preceq\bb$ is
finite for all $\bb \in \ZZ^n$.

The $\ZZ^n$-graded Betti numbers of any co-artinian monomial module
can be computed by taking the homology of its Koszul simplicial
complexes for degrees $\bb \in \ZZ^n$.

\begin{prop}[{\cite[Corollary~1.13]{bayer-sturmfels1998}}]\label{p:hochster}
The $i^\mathrm{th}$ Betti number of any co-artinian monomial module
$M$ in degree $\bb \in \ZZ^n$ is
\begin{align*}
\beta_{i,\bb} (M)
 &= \dim_\kk \tor_i^R(\kk, M)_\bb\\
 &= \dim_\kk \HH_{i-1}(K^\bb M;\kk),
\end{align*}
where $K^\bb M = \{ \tau \in \{0,1\}^n \mid \xx^{\bb - \tau} \in M\}$
is the \emph{Koszul simplicial complex} of~$M$ in degree~$\bb$ and
$\HH\!$ denotes reduced homology.
\end{prop}
\noindent
Thus in a minimal $\ZZ^n$-graded free resolution of $M$ over~$R$,
the $i^\mathrm{th}$ free module with basis in degree~$\bb$ can be
expressed~as
\begin{equation}\label{eq:res-as-module}
  F_{i,\bb} = \HH_{i-1}(K^\bb M;\kk) \otimes_\kk R(-\bb),
\end{equation}
where $N(-\bb)$ is the $\ZZ^n$-graded shift of any $R$-module~$N$ up
by~$\bb$, so $N(-\bb)_\aa = N_{\aa-\bb}$.  What remains is the central
problem: specify differentials in the free resolution over~$R$.

In fact, for the current purpose, where $M = M_L$ is a lattice module,
arbitrary differentials~$\del_i$ over~$R$ do not suffice: they must in
addition be $L$-equivariant, in the sense that $\del_i$ should commute
with translation by~$\ell$ for $\ell \in L$.  Equivalently, the
resolution should carry an action of the group algebra~$R[L]$, making
it an $R[L]$-free resolution of~$M_L$.  The reason is to be able to
quotient modulo the action of~$L$ to get a $\ZZ^n/L$-graded free
resolution of~$R/I_L$ by free $R$-modules
\cite[Corollary~3.3]{bayer-sturmfels1998}.

Once the module structure of an $R$-free resolution of a co-artinian
module~$M$ is specified by Eq.~(\ref{eq:res-as-module}), the sylvan
method \cite{sylvan} yields differentials: for each pair $\aa \preceq
\bb$ of comparable lattice points in~$\ZZ^n$, construct a \emph{sylvan
homology morphism} $\HH_{i-1} K^\aa M \from \HH_i K^\bb M$
(Definition~\ref{d:sylvan}) in such a way that the induced
homomorphisms
$$
  \HH_{i-1} K^\aa M \otimes \kk[\xx] (-\aa)
  \,\from\,
  \HH_i K^\bb M \otimes \kk[\xx] (-\bb)
$$
of $\ZZ^n$-graded free $R$-modules for all pairs $\aa \preceq \bb$
constitute a minimal free resolution \cite[Theorem~3.7 and
Corollary~9.5]{sylvan}.
Alas, the sylvan method was developed for monomial ideals~$I$ rather
than for monomial modules.  Our first observation is therefore that
the sylvan method extends with no difficulty to the case where the
monomial ideal~$I$ is replaced by a co-artinian monomial module~$M$
(Proposition~\ref{p:sylvan-for-M}).

When $M = M_L$ is a lattice module, the set of sylvan morphisms is
automatically $L$-invariant (Proposition~\ref{p:equivariant}).  For
our main result, Theorem~\ref{t:sylvan}, it remains only to take
the quotient modulo~$L$ of the sylvan $R$-free resolution of~$M_L$
thus constructed.

\section{The sylvan method}\label{s:sylvan}

The sylvan method \cite{sylvan} constructs explicit, closed-form
minimal free resolutions of a monomial ideal~$I$ from combinatorial
information \cite[Definitions~3.5 and~9.4]{sylvan} intrinsic to the
Koszul simplicial complexes~$K^\bb I$ for all lattice points $\bb \in
\ZZ^n$.  This procedure works when the field~$\kk$ has arbitrary
characteristic if certain combinatorial choices are allowed in each
Koszul simplicial complex \cite[Corollary~9.5]{sylvan}.  In addition,
the procedure works canonically, without making any choices at all, as
long as the characteristic of~$\kk$ avoids finitely many primes
\cite[Theorem~3.7]{sylvan}.  Beyond some easily stated basic properties
of the differentials in these sylvan resolutions---canonical or
non-canonical---the details are not relevant to the constructions of
sylvan resolutions of lattice ideals here.  We therefore isolate the
properties of sylvan resolutions required for the extension to
monomial modules and then to lattice ideals.

\begin{defn}[{\cite[Definition~2.1, Example~2.2, and Definition~9.2]{sylvan}}]%
\label{d:hedge}
Fix a simplicial complex $K$ with reduced
differential $\del_i: \wt C_i K \to \wt C_{i-1} K$ over a given
ring~$A$ which is assumed to be a field~$\kk$ or the integers~$\ZZ$.
\begin{enumerate}
\item\label{i:shrubbery}%
A \emph{shrubbery} for $\del_i$ is a maximal set~$T$ of
$i$-dimensional faces of~$K$ whose image $\del_i(T)$ is independent in
the boundaries $\wt B_i K = \del_i(\wt C_i K)$.

\item\label{i:stake}%
A \emph{stake set} for $\del_i$ is a minimal set~$S$ of
$(i-1)$-dimensional faces of~$K$
such that the composite $\wt B_i K \into \wt C_i K \onto A\{S\}$ is
injective.

\item\label{i:hedge}%
A \emph{hedge} for $\del_i$ is a pair
consisting of a stake set
and a shrubbery
for~$\del_i$.  A hedge for $\del_i$ may be expressed as $\st_i =
(S_{i-1},T_i)$ to indicate that the faces in~$S$ have dimension~$i-1$
while the faces in~$T$ have dimension~$i$.

\item\label{i:community}%
A \emph{community} in~$K$ is a sequence $\st_\spot =
(\st_0,\st_1,\st_2,\dots)$ of hedges for $\del_0, \del_1, \del_2\dots$
with $T_i \cap S_i = \nothing$ for all~$i$.
\end{enumerate}
\end{defn}

\begin{defn}\label{d:sylvan}
Fix a co-artinian monomial $R$-module~$M$.
A family of \mbox{homomorphisms}
$$
  \wt C_{i-1} K^\aa M
  \ \stackrel{\ D^{\aa\bb}}\filleftmap\
  \wt C_i K^\bb M
$$
between chain groups of the Koszul simplicial complexes of~$M$
over~$\kk$ for all comparable pairs $\aa \preceq \bb$ of lattice
points and all homological degrees~$i \in \ZZ$ is
\emph{canonical~sylvan}~if
\begin{enumerate}
\item\label{i:induced-homs}%
$D^{\aa\bb}$ induces
morphisms $\HH_{i-1} K^\aa M \from \HH_i K^\bb M$ whose induced
homomorphisms
$$
  \HH_{i-1} K^\aa M \otimes_\kk R(-\aa)
  \,\from\,
  \HH_i K^\bb M \otimes_\kk R(-\bb)
$$
of $\NN^n$-graded free $R$-modules constitute a minimal free
resolution of~$M$, and

\item\label{i:canonical}%
$D^{\aa\bb}$ depends only on the Koszul simplicial complexes $K^\cc M$
indexed by $\cc$ in the interval $[\aa,\bb]$.
\end{enumerate}\setcounter{separated}{\value{enumi}}
The family $D^{\aa\bb}$ is \emph{noncanonical sylvan} if, instead of
condition~\ref{i:canonical},
\begin{enumerate}\setcounter{enumi}{\value{separated}}
\item\label{i:noncanonical}%
$D^{\aa\bb}$ depends only on the Koszul simplicial complexes $K^\cc
M$, along with a community therein, indexed by $\cc$ in the interval
$[\aa,\bb]$.
\end{enumerate}
The family $\{D^{\aa\bb}\}_{\aa\preceq\bb}$ is \emph{sylvan} if it is
canonical or noncanonical sylvan.
\end{defn}

In other words, in a canonical sylvan resolution of a monomial
module~$M$, the homomorphism $\HH_{i-1} K^\aa M \otimes R(-\aa)
\,\from\, \HH_i K^\bb M \otimes R(-\bb)$ is constructed entirely from
information intrinsic to the restriction of~$M$ to the interval
$[\aa-\1,\bb]$, where $\1 = (1,\dots,1)$, with no external choices;
the $\1$ is needed because $K^\cc M$ reflects the interval $[\cc - \1,
\cc]$.  In contrast, noncanonical sylvan resolutions involve choices
of communities.

\begin{example}\label{e:noncanonical-sylvan}
Any monomial ideal $I \subseteq R$ has a noncanonical sylvan family
\cite[Corollary~9.5]{sylvan} in which $D^{\aa\bb}$ is specified by an
explicit, closed-form sum over all saturated decreasing lattice paths
from $\bb$ to~$\aa$ \cite[Definition~9.4]{sylvan}, once communities
have been specified in the Koszul simplicial complexes of~$I$; any
communities suffice.
\end{example}

To specify restrictions on the characteristic of the field~$\kk$ in
canonical sylvan constructions, we
summarize the relevant points from \cite[Definitions~2.1, 2.11,
and~2.13]{sylvan}.

\begin{defn}\label{d:torsionless}
Fix a simplicial complex~$K$.
\begin{enumerate}
\item%
Write $\tau(T)$ for the index of the subgroup of~$\wt B_i$ generated
over~$\ZZ$ by the images of the faces in a shrubbery~$T$ for~$\del_i$
over~$\ZZ$.

\item%
Dually, write $\sigma(S)$ for the index of the image of~$\wt B_i K$ in
the span~$\ZZ\{S\}$ of a stake set~$S$ for~$\del_i$ over~$\ZZ$.
\end{enumerate}
A field~$\kk$ is \emph{torsionless} for~$K$ if for all~$i$ the
characteristic of~$\kk$ does not divide $\tau_i = \sum_T \tau(T)^2$ or
$\sigma_i = \sum_S \sigma(S)^2$ or the order of the torsion subgroup
of~$\wt C_i K / \wt B_i K$, where the sums are respectively over all
shrubberies~$T$ and stake sets~$S$ for~$\del_i$.
\end{defn}

\begin{example}\label{e:canonical-sylvan}
Any monomial ideal $I \subseteq R$ has a canonical sylvan family
\cite[Theorem~3.7]{sylvan} in which $D^{\aa\bb}$ is specified by an
explicit, closed-form sum over all saturated decreasing lattice paths
from $\bb$ to~$\aa$ \cite[Definition~3.6]{sylvan}, once $\kk$ is
torsionless
for the
Koszul simplicial complexes of~$I$.
\end{example}

\begin{remark}\label{r:torsionless}
There are only finitely many simplicial complexes on $n$~vertices, so
a field~$\kk$ is torsionless for an arbitrary co-artinian monomial
module~$M$ it its characteristic avoids finitely many positive integer
primes.
\end{remark}

The main purpose of this is section is to extend the sylvan and
canonical sylvan families in Examples~\ref{e:noncanonical-sylvan}
and~\ref{e:canonical-sylvan} to co-artinian monomial modules.  The key
but nonetheless elementary observation is that sylvan morphisms are
preserved by translation.

\begin{lemma}\label{l:translation}
Fix a monomial module~$M$ with a sylvan family (canonical or
noncanonical).  Translation by a vector~$\ell \in \ZZ^n$ naturally
induces a sylvan family on the $\ZZ^n$-graded shift~$M(-\ell)$ in
which the sylvan homology morphism
$$
  \wt C_{i-1} K^{\aa+\ell}M(-\ell)
  \ \stackrel{\ D^{\aa+\ell\;\bb+\ell}}\filleftmap\
  \wt C_i K^{\bb+\ell}M(-\ell)
$$
is the sylvan homology morphism
$D^{\aa\bb}$ on~$M$ itself.
\end{lemma}
\begin{proof}
Translation up by $\ell$ takes that interval to $[\aa + \ell, \bb +
\ell]$.  The lemma is therefore immediate from
Definition~\ref{d:sylvan} and the fact that $K^{\cc+\ell} M(-\ell) =
K^\cc M$.
\end{proof}

This observation has two important manifestations, detailed in
Propositions~\ref{p:sylvan-for-M} and~\ref{p:equivariant}, which
respectively use translation of
\begin{itemize}
\item%
any co-artinian module up so that an arbitrarily large subset of it
sits in the nonnegative orthant, and
\item%
a lattice module along the lattice.
\end{itemize}

\begin{prop}\label{p:sylvan-for-M}
Fix a co-artinian monomial module~$M$.  The noncanonical and canonical
sylvan families in Examples~\ref{e:noncanonical-sylvan}
and~\ref{e:canonical-sylvan} work verbatim when the monomial ideal~$I$
is replaced by an arbitrary co-artinian monomial module~$M$.
\end{prop}
\begin{proof}
The co-artinian hypothesis implies that only finitely many free
modules
$$
  \HH_{i-1} K^\aa M \otimes R(-\aa)
$$
contribute nonzero $\ZZ^n$-graded degree~$\bb$ components.  That the
displayed homomorphisms in
Definition~\ref{d:sylvan}.\ref{i:induced-homs} constitute a complex
and that this complex is exact can hence be verified at any given
$\ZZ^n$-graded degree~$\bb$ by translating the monomial module~$M$ up
so that the nonnegative orthant contains all of the degrees~$\aa$
beneath~$\bb$ such that $K^\aa M$ has at least one face.
\end{proof}

\section{Equivariant sylvan resolutions}\label{s:equivariant}

\begin{lemma}\label{l:equivariant}
When $M = M_L$ is a co-artinian lattice module, the Koszul simplicial
complexes $K^\aa M_L$ and $K^{\aa + \ell} M_L$ are equal as simplicial
subcomplexes of the simplex on $\{1,\dots,n\}$ whenever $\ell \in L$.
\end{lemma}
\begin{proof}
This is immediate from the $L$-invariance of~$M_L$ itself.
\end{proof}

\begin{remark}\label{r:equivariant}
When $M = M_L$ is a lattice module in
Proposition~\ref{p:sylvan-for-M}, any canonical sylvan family is
automatically $L$-equivariant, in a sense to be made precise in
Proposition~\ref{p:equivariant}.  However, noncanonical sylvan
families need not be $L$-equivariant: different communities can be
selected in Koszul simplicial complexes indexed by lattice points in
the same coset of~$L$ even though the Koszul simplicial complexes
themselves are the same.  Equivariant noncanonical sylvan families
still exist, which is crucial for constructing closed-form
combinatorial minimal resolutions over fields of arbitrary
characteristic, but additional care is required to construct them.
\end{remark}

\begin{prop}\label{p:equivariant}
Fix a co-artinian lattice module~$M_L$.
\begin{enumerate}
\item%
In any canonical sylvan family the maps satisfy $D^{\aa\bb} =
D^{\aa+\ell\;\bb+\ell}$ for all $\ell \in L$ as homomorphisms from
chains of~$K^\bb M_L\hspace{-.07pt} = K^{\bb+\ell} M_L$ to those of
$K^\aa M_L\hspace{-.07pt} = \nolinebreak K^{\aa+\ell} M_L$.%

\item\label{i:equivariant-noncanonical}%
A noncanonical sylvan family can be constructed so that $D^{\aa\bb} =
D^{\aa+\ell\;\bb+\ell}$ for all $\ell \in L$ by selecting one
community in each of the Koszul simplicial complexes~$K^\cc M_L$ for
$\cc$ in a set of representatives for the cosets of~$L$.
\end{enumerate}
These sylvan families are called \emph{equivariant}.
\end{prop}
\begin{proof}
The canonical case follows from Lemma~\ref{l:translation} by
Definition~\ref{d:sylvan}.\ref{i:canonical}.

For the noncanonical case, copy the selected community for each coset
representative into the Koszul simplicial complexes for all other
lattice points in the coset.  This yields a sylvan family as in
Example~\ref{e:noncanonical-sylvan} by way of
Proposition~\ref{p:sylvan-for-M}.
\end{proof}

\begin{cor}\label{c:equivariant}
The minimal free resolution over~$R$ of a co-artinian lattice
module~$M_L$ arising from an equivariant sylvan family
(Proposition~\ref{p:equivariant}) is $L$-equivariant; equivalently, it
is a minimal $\ZZ^n$-graded $R[L]$-free resolution of~$M_L$.\qed
\end{cor}

Now comes the main result: the construction of combinatorial,
closed-form minimal free resolutions of co-artinian lattice ideals in
arbitrary characteristic, and canonical such resolutions when the
characteristic avoids finitely many positive primes.

\begin{thm}\label{t:sylvan}
Fix a field\/~$\kk$ and a lattice ideal $I_L \subseteq R =
\kk[x_1,\dots,x_n]$.
\begin{enumerate}
\item\label{i:thm-canonical}%
If\/~$\kk$ is torsionless (Definition~\ref{d:torsionless}) for the
Koszul simplicial complexes of the lattice module~$M_L$ and $F_\spot$
is the free resolution of the lattice module~$M_L$ afforded by the
canonical sylvan family in Proposition~\ref{p:sylvan-for-M}, then
$F_\spot \otimes_{R[L]} R$ is a minimal $\ZZ^n/L$-graded $R[L]$-free
resolution of~$M_L$.

\item\label{i:thm-noncanonical}%
If $F_\spot$ is the free resolution of the lattice module~$M_L$
afforded by any noncanonical sylvan family as in
Proposition~\ref{p:equivariant}.\ref{i:equivariant-noncanonical}, then
$F_\spot \otimes_{R[L]} R$ is a minimal $\ZZ^n/L$-graded $R[L]$-free
resolution of~$M_L$.
\end{enumerate}
\end{thm}
\begin{proof}
By Corollary~\ref{c:equivariant} the Bayer--Sturmfels functor
\cite[Corollary~3.3]{bayer-sturmfels1998} applies.
\end{proof}

\begin{remark}\label{r:sylvan}
Unwinding the definitions leading to Theorem~\ref{t:sylvan} helps
exhibit the explicit nature of the minimal resolutions it constructs.
For $\aalpha \preceq \bbeta \in \ZZ^n/L$, choose coset representatives
$\aa \preceq \bb \in \ZZ^n$.  To $\aalpha$ and~$\bbeta$ are associated
Koszul simplicial complexes $K^\aalpha L = K^\aa M_L$ and $K^\bbeta L
=\nolinebreak K^\bb M_L$ on~$\{1,\dots,n\}$.  Any sylvan resolution
of~$I_L$, be it canonical or noncanonical, is expressed by specifying
$\ZZ^n/L$-graded homomorphisms
$$
  \HH_{i-1} K^\aalpha L \otimes_\kk R(-\aalpha)
  \,\from\,
  \HH_i K^\bbeta L \otimes_\kk R(-\bbeta).
$$
These are induced by sylvan homology morphisms $\HH_{i-1} K^\aalpha L
\from \HH_i K^\bbeta L$ which are explicitly enacted on (Koszul
simplicial) cycles by sylvan morphisms
$$
  \wt C_{i-1} K^\aalpha L
  \ \stackrel{\ D^{\aalpha\bbeta}}\filleftmap\
  \wt C_i K^\bbeta L.
$$
Each $D^{\aalpha\bbeta}$ is given by its \emph{sylvan matrix}, whose
rows and columns are indexed by faces (i.e., by subsets of
$\{1,\dots,n\}$).  To be completely precise requires notions from
\cite{sylvan}, which we use henceforth without further comment.
\begin{enumerate}
\item%
In the canonical sylvan free resolution of~$I_L$
the entry of $D_{\sigma\tau}$ indexed by faces $\sigma \in
K_{i-1}^\aalpha L$ and $\tau \in\nolinebreak K_i^\bbeta L$ is a
normalized sum of the weights~$w_\phi$ of all chain-link fences~$\phi$
from~$\tau$ to~$\sigma$ along all saturated decreasing lattice
paths~$\lambda$ from~$\bb$~to~$\aa$:
$$
  D_{\sigma\tau}
  =
  \sum_{\lambda \in \Lambda(\aa,\bb)}
  \frac{1}{\Delta_{i,\lambda} M_L}
  \sum_{\phi \in \Phi_{\sigma\tau}(\lambda)}
  w_\phi.
$$

\item%
A noncanonical sylvan free resolution of~$I_L$ requires a choice of
community (Definition~\ref{d:hedge}.\ref{i:community}) in the Koszul
simplicial complex~$K^\aalpha L$ for each coset $\aalpha \in \ZZ^n/L$.
The entry of $D_{\sigma\tau}$ indexed by faces $\sigma \in
K_{i-1}^\aalpha L$ and $\tau \in\nolinebreak K_i^\bbeta L$ is the sum,
over all saturated decreasing lattice paths~$\lambda$ from~$\bb$
to~$\aa$, of the weights~$w^\kk_\phi$ over~$\kk$ of all chain-link
fences~$\phi$ from~$\tau$ to~$\sigma$ that are subordinate to the
hedgerow~$ST_i^\lambda$ derived from the communities along~$\lambda$:
$$
  D_{\sigma\tau}
  =
  \sum_{\lambda \in \Lambda(\aa,\bb)}\
  \sum_{\substack{\phi \in \Phi_{\sigma\tau}(\lambda)\\\phi\vdash ST_i^\lambda}}
  w^\kk_\phi.
$$
\end{enumerate}
\end{remark}
\begin{excise}{%
  \begin{example}
  The lattice $L$ that is the kernel of matrix $[4\,\,3\,\,5]$ has
  lattice ideal $I_L = \<xy^2 - z^2, xz - y^3, yz - x^2\> \subseteq S =
  \kk[x, y, z]$.  The corresponding lattice module is $M_L = \<x^a y^b
  z^c\mid 4a + 3b + 5c = 0\> \subseteq \kk[x^{\pm 1}, y^{\pm 1}, z^{\pm
  1}]$.  Note that $M_L$ is a $\ZZ^3$-graded $S[L]$-submodule of
  $\kk[x^{\pm 1}, y^{\pm 1}, z^{\pm 1}]$.  The hull resolution of~$M_L$
  is constructed in \cite[Example~9.21]{cca}, which also contains a
  picture of~$M_L$.
  \\[2ex]
  \comment{EM: I will consider putting a version of this picture in here
  if this example is modified to warrant it.  As it is, the example
  lists a bunch of sylvan matrices but doesn't illustrate how
  Theorem~\ref{t:sylvan} works.  To be honest, I'm a little confused as
  to why there are $u,v,w$ variables in some of the sylvan data (see my
  response to \#1 in my email dated Sep 16); that has to be explained.
  Taking a broader view, seeing all of the sylvan matrices may be less
  informative than seeing a few sample sylvan computations on
  $L$-equivalent pairs of degrees which collated yield a selected one of
  the depicted equivariant sylvan matrices.  For example, there are
  paths from 120 to 000 and also paths from 120 to 12$\{$-2$\}$, which
  presumably give rise to the top two sylvan matrices; how does the
  tensor product with~$R$ over $R[L]$ in Theorem~\ref{t:sylvan} combine
  these two?  What about the paths from $\{$-1$\}$31 to $\{$-2$\}$11?
  It might be worth including the brass-tacks hedge and fence
  computations that lead to the entries of the sylvan matrices in these
  sample computations---say, one for $\phi_1$ and one for~$\phi_2$.
  After these lower-level details are carried out in sample situations,
  it might then be worth listing all of the sylvan matrices, as is
  currently done.  It could also be useful to show readers how to notate
  the quotient resolution of~$I_L$ in a sylvan kind of way, if that's
  possible.}
  \\[2ex]
  Our computation of sylvan morphisms is based on the explicit formulas
  for three variables case in [add citation].
  Writing down the resolution requires specifying sylvan data for the
  homomorphisms
  $$
  \begin{array}{@{}c@{\,}}
     \CC_{{\!-\!}1} K^{000} 
  \end{array}
  \xleftarrow{\phi_1}
  {\begin{array}{@{}c@{\,}}
     \CC_{\!0} K^{120} 
  \\ \oplus
  \\ \CC_{\!0} K^{101} 
  \\ \oplus
  \\ \CC_{\!0} K^{011}  
  \end{array}
  }
  \xleftarrow{\phi_2}
  {\begin{array}{@{}c@{\,}}
  \CC_{1}K^{130} 
  \\ \oplus
  \\ \CC_1 K^{102} 
  \end{array}}
  $$
  The sylvan morphisms for $\phi_1$ arise from six pairs of Koszul
  simplicial complexes.  For notation, the variables $u$, $v$, and~$w$
  are utilized to express monomials in~$\ZZ[L]$.
  \begin{align*}
  \CC_{{\!-\!}1} K^{000} \otimes \<1\>
  &
  \xleftarrow{
    \monomialmatrix
      {\nothing}
      {\scriptstyle
       \begin{array}{@{\,}c@{\ \,}c@{\ \,}c@{\,}}
       x & y & z
       \\
       1 & 1 & 0
       \end{array}
      }
      {\\}
  }
  \CC_0 K^{120} \otimes \<xy^2\>
  \\
  \CC_{{\!-\!}1} K^{000} \otimes \<1\>
  &
  \xleftarrow{
    \monomialmatrix
      {\nothing}
      {\scriptstyle
       \begin{array}{@{\,}c@{\ \,}c@{\ \,}c@{\,}}
       x & y & z
       \\
       0 & 0 & 1
       \end{array}
      }
      {\\}
  }
  \CC_0 K^{120} \otimes \<z^2uv^2w^{-2}\>
  \\
  \CC_{{\!-\!}1} K^{000} \otimes \<1\>
  &
  \xleftarrow{
    \monomialmatrix
      {\nothing}
      {\scriptstyle
       \begin{array}{@{\,}c@{\ \,}c@{\ \,}c@{\,}}
       x & y & z
       \\
       1 & 0 & 1
       \end{array}
      }
      {\\}
  }
  \CC_0 K^{101} \otimes \<xz\>
  \\
  \CC_{{\!-\!}1} K^{000} \otimes \<1\>
  &
  \xleftarrow{
    \monomialmatrix
      {\nothing}
      {\scriptstyle
       \begin{array}{@{\,}c@{\ \,}c@{\ \,}c@{\,}}
       x & y & z
       \\
       0 & 1 & 0
       \end{array}
      }
      {\\}
  }
  \CC_0 K^{101} \otimes \<y^3uv^{-3}w\>
  \\
  \CC_{{\!-\!}1} K^{000} \otimes \<1\>
  &
  \xleftarrow{
    \monomialmatrix
      {\nothing}
      {\scriptstyle
       \begin{array}{@{\,}c@{\ \,}c@{\ \,}c@{\,}}
       x & y & z
       \\
       0 & 1 & 1
       \end{array}
      }
      {\\}
  }
  \CC_0 K^{011} \otimes \<yz\>
  \\
  \CC_{{\!-\!}1} K^{000} \otimes \<1\>
  &
  \xleftarrow{
    \monomialmatrix
      {\nothing}
      {\scriptstyle
       \begin{array}{@{\,}c@{\ \,}c@{\ \,}c@{\,}}
       x & y & z
       \\
       1 & 0 & 0
       \end{array}
      }
      {\\}
  }
  \CC_0 K^{011} \otimes \<x^2u^{-2}vw\>
  \end{align*}
  The sylvan morphisms for $\phi_2$ arise from six pairs of Koszul
  simplicial complexes.
  \begin{align*}
  \CC_0 K^{120} \otimes \<xy^2\>
  &
  \xleftarrow{
    \monomialmatrix
      {x\\ y\\ z}
      {\begin{array}{@{\ }c@{\ \,}r@{\ \,}r@{\,}}
           zy &        yx        &        xz\,
         \\
            0 & \nicefrac{ 1}{3} & \nicefrac{-1}{3}\,
         \\ 0 & \nicefrac{-2}{3} & \nicefrac{-1}{3}\,
         \\ 0 & \nicefrac{ 1}{3} & \nicefrac{ 2}{3}\,
       \end{array}
      }
      {\\\\\\}
  }
  \CC_1 K^{102} \otimes \<x^2y^2u^{-1}v^{-2}w^2\>
  \\
  \CC_0 K^{101} \otimes \<xz\>
  &
  \xleftarrow{
    \monomialmatrix
     {x\\ y\\ z}
      {\begin{array}{@{}r@{\ \ \;}c@{\ }r@{\,}}
                  zy        &yx &        xz\,
       \\\,\nicefrac{-1}{3} & 0 & \nicefrac{-2}{3}\,
       \\\,\nicefrac{ 2}{3} & 0 & \nicefrac{ 1}{3}\,
       \\\,\nicefrac{-1}{3} & 0 & \nicefrac{ 1}{3}\,
       \end{array}
      }
      {\\\\\\}
  }
  \CC_1 K^{102} \otimes \<xz^2\>
  \\
  \CC_0 K^{011} \otimes \<x^2 u^{-2}vw\>
  &
  \xleftarrow{
    \monomialmatrix
     {x\\ y\\ z}
      {\begin{array}{@{}r@{\ \,}r@{\ \ \,}c@{\,}}
                  zy        &        yx        &xz\,
       \\\,\nicefrac{ 4}{9} & \nicefrac{ 5}{9} & 0\,
       \\\,\nicefrac{ 1}{9} & \nicefrac{-1}{9} & 0\,
       \\\,\nicefrac{-5}{9} & \nicefrac{-4}{9} & 0\,
       \end{array}
      }
      {\\\\\\}
  }
  \CC_1 K^{102} \otimes \<x^2y^2u^{-1}v^{-2}w^2\>
  \\
  \CC_0 K^{120} \otimes \<xy^2\>
  &
  \xleftarrow{
    \monomialmatrix
     {x\\ y\\ z}
      {\begin{array}{@{}r@{\ \,}r@{\ \ \,}c@{\,}}
                  zy        &        yx        &xz\,
       \\\,\nicefrac{ 1}{3} & \nicefrac{ 2}{3} & 0\,
       \\\,\nicefrac{ 1}{3} & \nicefrac{-1}{3} & 0\,
       \\\,\nicefrac{-2}{3} & \nicefrac{-1}{3} & 0\,
       \end{array}
      }
      {\\\\\\}
  }
  \CC_1 K^{130} \otimes \<xy^3\>
  \\
  \CC_0 K^{101} \otimes \<y^3uv^{-3}w\>
  &
  \xleftarrow{
    \monomialmatrix
     {x\\ y\\ z}
      {\begin{array}{@{\ }c@{\ \,}r@{\ \,}r@{\,}}
          zy &        yx        &        xz\,
       \\\,0 & \nicefrac{ 1}{3} & \nicefrac{-1}{3}\,
       \\\,0 & \nicefrac{-2}{3} & \nicefrac{-1}{3}\,
       \\\,0 & \nicefrac{ 1}{3} & \nicefrac{ 2}{3}\,
       \end{array}
      }
      {\\\\\\}
  }
  \CC_1 K^{130} \otimes \<xy^3\>
  \\
  \CC_0 K^{011} \otimes \<yz\>
  &
  \xleftarrow{
    \monomialmatrix
     {x\\ y\\ z}
      {\begin{array}{@{}r@{\ \ \;}c@{\ }r@{\,}}
                  zy        &yx &        xz\,
       \\\,\nicefrac{-1}{3} & 0 & \nicefrac{-2}{3}\,
       \\\,\nicefrac{ 2}{3} & 0 & \nicefrac{ 1}{3}\,
       \\\,\nicefrac{-1}{3} & 0 & \nicefrac{ 1}{3}\,
       \end{array}
      }
      {\\\\\\}
  }
  \CC_1 K^{130} \otimes \<yz^2uv^2w^{-2}\>
  \end{align*}
  \end{example}
}\end{excise}%
%

\section{Extended example}\label{s:example}

The lattice $L$ that is the kernel of the matrix $[4\,\,3\,\,5]$ has
lattice ideal
$$
  I_L = \<xy^2 - z^2, xz - y^3, yz - x^2\> \subseteq R = \kk[x, y, z].
$$
The corresponding lattice module is
$$
  M_L
  =
  \<x^a y^b z^c\mid 4a + 3b + 5c = 0\>
  \subseteq
  \kk[x^{\pm 1}, y^{\pm 1}, z^{\pm 1}].
$$
Note that $M_L$ is a $\ZZ^3$-graded $R[L]$-submodule of $\kk[x^{\pm
1}, y^{\pm 1}, z^{\pm 1}]$.  The hull resolution of~$M_L$
is constructed in \cite[Example~9.21]{cca}, which also contains a
picture of~$M_L$.

For the moment, treat $M_L$ as an $R$-module with generators
$\{x^{l_1}y^{l_2}z^{l_3} \mid (l_1,l_2,l_3) \in\nolinebreak L\}$.  The
first goal is to list all possible sylvan morphisms between Koszul
simplicial complexes up to translation by vectors in~$L$.  Our
computation of sylvan morphisms is based on the explicit formulas for
three variables in \cite[Theorem~3.1 and Theorem~4.1]{ordog2020}.
There are six different nonzero sylvan morphisms from $\CC_0$ to
$\CC_{-1}$:
\begin{align*}
\quad\
\CC_{{\!-\!}1} K^{000} \otimes \<1\>
&\mkl{\hspace{-26ex}(a)}
\xleftarrow{
  \monomialmatrix
    {\nothing}
    {\scriptstyle
     \begin{array}{@{\,}c@{\ \,}c@{\ \,}c@{\,}}
     x & y & z
     \\
     1 & 1 & 0
     \end{array}
    }
    {\\}
}
\CC_0 K^{120} \otimes \<xy^2\>
\quad
\smash{
  \begin{array}[t]{@{}c@{}}
  \\[-7ex]
    \psfrag{a}{\tiny$\blu a$}
    \psfrag{b}{\tiny$\blu b$}
    \includegraphics[width=27ex]{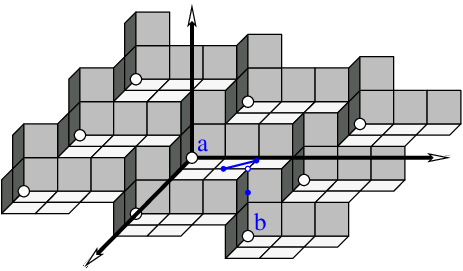}
  \end{array}
}
\\[-.5ex]
\quad\
\CC_{{\!-\!}1} K^{12\{-2\}} \otimes \<xyz^{-2}\>
&\mkl{\hspace{-26ex}(b)}
\xleftarrow{
  \monomialmatrix
    {\nothing}
    {\scriptstyle
     \begin{array}{@{\,}c@{\ \,}c@{\ \,}c@{\,}}
     x & y & z
     \\
     0 & 0 & 1
     \end{array}
    }
    {\\}
}
\CC_0 K^{120} \otimes \<xy^2\>
\\[-.5ex]
\quad\
\CC_{{\!-\!}1} K^{000} \otimes \<1\>
&\mkl{\hspace{-26ex}(c)}
\xleftarrow{
  \monomialmatrix
    {\nothing}
    {\scriptstyle
     \begin{array}{@{\,}c@{\ \,}c@{\ \,}c@{\,}}
     x & y & z
     \\
     1 & 0 & 1
     \end{array}
    }
    {\\}
}
\CC_0 K^{101} \otimes \<xz\>
\quad
\smash{
  \begin{array}[t]{@{}c@{}}
  \\[-7ex]
    \psfrag{a}{\tiny$\blu c$}
    \psfrag{b}{\tiny$\blu d$}
    \includegraphics[width=27ex]{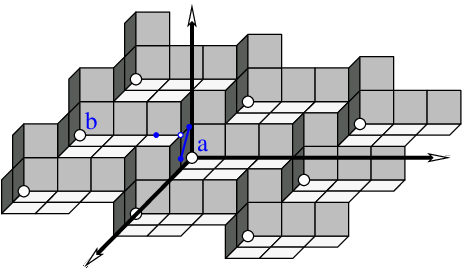}
  \end{array}
}
\\[-.5ex]
\quad\
\CC_{{\!-\!}1} K^{1\{-3\}1} \otimes \<xy^{-3}z\>
&\mkl{\hspace{-26ex}(d)}
\xleftarrow{
  \monomialmatrix
    {\nothing}
    {\scriptstyle
     \begin{array}{@{\,}c@{\ \,}c@{\ \,}c@{\,}}
     x & y & z
     \\
     0 & 1 & 0
     \end{array}
    }
    {\\}
}
\CC_0 K^{101} \otimes \<xz\>
\\[-.5ex]
\quad\
\CC_{{\!-\!}1} K^{000} \otimes \<1\>
&\mkl{\hspace{-26ex}(e)}
\xleftarrow{
  \monomialmatrix
    {\nothing}
    {\scriptstyle
     \begin{array}{@{\,}c@{\ \,}c@{\ \,}c@{\,}}
     x & y & z
     \\
     0 & 1 & 1
     \end{array}
    }
    {\\}
}
\CC_0 K^{011} \otimes \<yz\>
\quad
\smash{
  \begin{array}[t]{@{}c@{}}
  \\[-7ex]
    \psfrag{a}{\tiny$\blu e$}
    \psfrag{b}{\tiny$\blu f$}
    \includegraphics[width=27ex]{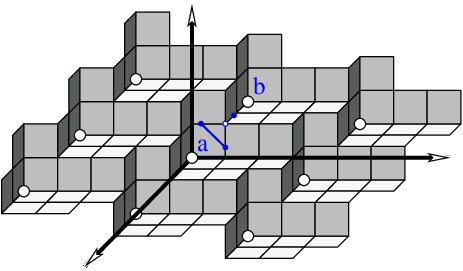}
  \end{array}
}
\\[-.5ex]
\quad\
\CC_{{\!-\!}1} K^{\{-2\}11} \otimes \<x^{-2}yz\>
&\mkl{\hspace{-26ex}(f)}
\xleftarrow{
  \monomialmatrix
    {\nothing}
    {\scriptstyle
     \begin{array}{@{\,}c@{\ \,}c@{\ \,}c@{\,}}
     x & y & z
     \\
     1 & 0 & 0
     \end{array}
    }
    {\\}
}
\CC_0 K^{011} \otimes \<yz\>.
\end{align*}
There are also six different nonzero sylvan morphisms from $\CC_1$
to~$\CC_0$:
\begin{align*}
\CC_0 K^{002} \otimes \<z^2\>
&\mkl{\hspace{-25.5ex}(a)}
\xleftarrow{
  \monomialmatrix
   {x\\ y\\ z}
    {\begin{array}{@{}r@{\ \,}r@{\ \ \,}c@{\,}}
                zy        &        yx        &xz\,
     \\\,\nicefrac{ 1}{3} & \nicefrac{ 2}{3} & 0\,
     \\\,\nicefrac{ 1}{3} & \nicefrac{-1}{3} & 0\,
     \\\,\nicefrac{-2}{3} & \nicefrac{-1}{3} & 0\,
     \end{array}
    }
    {\\\\\\}
}
\CC_1 K^{012} \otimes \<yz^2\>
\\
\qquad
\CC_0 K^{\{-1\}12} \otimes \<x^{-1}yz^2\>
&\mkl{\hspace{-25.5ex}(b)}
\xleftarrow{
  \monomialmatrix
   {x\\ y\\ z}
    {\begin{array}{@{\ }c@{\ \,}r@{\ \,}r@{\,}}
        zy &        yx        &        xz\,
     \\\,0 & \nicefrac{ 1}{3} & \nicefrac{-1}{3}\,
     \\\,0 & \nicefrac{-2}{3} & \nicefrac{-1}{3}\,
     \\\,0 & \nicefrac{ 1}{3} & \nicefrac{ 2}{3}\,
     \end{array}
    }
    {\\\\\\}
}
\CC_1 K^{012} \otimes \<yz^2\>
\hspace{-3.75ex}
\smash{
  \begin{array}[t]{@{}c@{}}
  \\[-18ex]
    \psfrag{a}{\tiny$\red a$}
    \psfrag{b}{\tiny$\red b$}
    \psfrag{c}{\tiny$\red c$}
    \includegraphics[width=27ex]{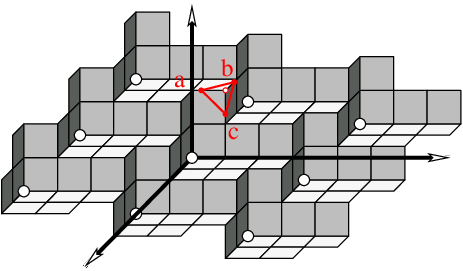}
  \end{array}
}
\\
\CC_0 K^{011} \otimes \<yz\>
&\mkl{\hspace{-25.5ex}(c)}
\xleftarrow{
  \monomialmatrix
   {x\\ y\\ z}
    {\begin{array}{@{}r@{\ \ \;}c@{\ }r@{\,}}
                zy        &yx &        xz\,
     \\\,\nicefrac{-1}{3} & 0 & \nicefrac{-2}{3}\,
     \\\,\nicefrac{ 2}{3} & 0 & \nicefrac{ 1}{3}\,
     \\\,\nicefrac{-1}{3} & 0 & \nicefrac{ 1}{3}\,
     \end{array}
    }
    {\\\\\\}
}
\CC_1 K^{012} \otimes \<yz^2\>
\\
\CC_0 K^{011} \otimes \<yz\>
&\mkl{\hspace{-26ex}(d)}
\xleftarrow{
  \monomialmatrix
    {x\\ y\\ z}
    {\begin{array}{@{\ }c@{\ \,}r@{\ \,}r@{\,}}
         zy &        yx        &        xz\,
       \\
          0 & \nicefrac{ 1}{3} & \nicefrac{-1}{3}\,
       \\ 0 & \nicefrac{-2}{3} & \nicefrac{-1}{3}\,
       \\ 0 & \nicefrac{ 1}{3} & \nicefrac{ 2}{3}\,
     \end{array}
    }
    {\\\\\\}
}
\CC_1 K^{031} \otimes \<y^3z\>
\\
\CC_0 K^{\{-1\}31} \otimes \<x^{-1}y^3z\>
&\mkl{\hspace{-26ex}(e)}
\xleftarrow{
  \monomialmatrix
   {x\\ y\\ z}
    {\begin{array}{@{}r@{\ \ \;}c@{\ }r@{\,}}
                zy        &yx &        xz\,
     \\\,\nicefrac{-1}{3} & 0 & \nicefrac{-2}{3}\,
     \\\,\nicefrac{ 2}{3} & 0 & \nicefrac{ 1}{3}\,
     \\\,\nicefrac{-1}{3} & 0 & \nicefrac{ 1}{3}\,
     \end{array}
    }
    {\\\\\\}
}
\CC_1 K^{031} \otimes \<y^3z\>
\hspace{-3.75ex}
\smash{
  \begin{array}[t]{@{}c@{}}
  \\[-18ex]
    \psfrag{a}{\tiny$\red d$}
    \psfrag{b}{\tiny$\red e$}
    \psfrag{c}{\tiny$\red f$}
    \includegraphics[width=27ex]{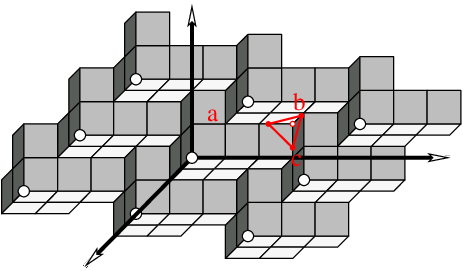}
  \end{array}
}
\\
\CC_0 K^{030} \otimes \<y^3\>
&\mkl{\hspace{-26ex}(f)}
\xleftarrow{
  \monomialmatrix
   {x\\ y\\ z}
    {\begin{array}{@{}r@{\ \,}r@{\ \ \,}c@{\,}}
                zy        &        yx        &xz\,
     \\\,\nicefrac{ 4}{9} & \nicefrac{ 5}{9} & 0\,
     \\\,\nicefrac{ 1}{9} & \nicefrac{-1}{9} & 0\,
     \\\,\nicefrac{-5}{9} & \nicefrac{-4}{9} & 0\,
     \end{array}
    }
    {\\\\\\}
}
\CC_1 K^{031} \otimes \<y^3z\>.
\end{align*}
In the staircase diagrams, the relevant Koszul simplicial complexes
are drawn (in color, where available), with the source degree at the
center of the simplicial complex and the target degrees labeled with
letters corresponding to the order in which the sylvan morphisms are
listed.  In the first list of six, each Koszul simplicial complex
accounts for two sylvan morphisms that have the same source degree and
different target degrees.  In the second list of six, each Koszul
simplicial complex accounts for three sylvan morphisms that have the
same source degree and different target degrees.

All other nonzero sylvan morphisms are translations by lattice points
in~$L$ of these twelve sylvan morphisms.  By
Proposition~\ref{p:equivariant}, all nonzero sylvan matrices are
identical to one of these twelve matrices, the difference residing
only in their $\ZZ^3$-graded degrees.  For instance, the sylvan
morphism
\begin{align*}
  \CC_0K^{13\{-1\}} &\from \CC_1 K^{130}
\intertext{agrees with the sylvan morphism~(c) in the second group of
six, namely}
     \CC_0K^{011}   &\from \CC_1 K^{012}
\end{align*}
from $\CC_1$ to~$\CC_0$, since $(1,3,0) - (0,1,2) = (1,3,-1) - (0,1,1)
= (1,2,-2) \in L$:
$$
\psfrag{a}{}
\psfrag{b}{}
\psfrag{c}{}
\psfrag{L}{\tiny$\in\!L$}
\includegraphics[width=38.5ex]{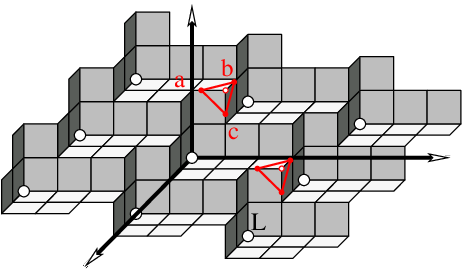}
$$

The $L$-invariance of these sylvan morphisms implies that the sylvan
resolution of~$M_L$ as an $R$-module is endowed with the additional
structure of an $R[L]$-free resolution of~$M_L$.  Picking a
representative for each relevant coset of~$L$ in~$\ZZ^3/L$ yields the
resolution
$$
\begin{array}{@{}c@{\,}}
   \CC_{{\!-\!}1} K^{000} 
\end{array}
\xleftarrow{\phi_1}
{\begin{array}{@{}c@{\,}}
   \CC_{\!0} K^{002} 
\\ \oplus
\\ \CC_{\!0} K^{\{-1\}12} 
\\ \oplus
\\ \CC_{\!0} K^{011}  
\end{array}
}
\xleftarrow{\phi_2}
{\begin{array}{@{}c@{\,}}
\CC_1 K^{012} 
\\ \oplus
\\ \CC_1 K^{031} 
\end{array}}
$$
of~$M_L$ over~$R[L]$, where $\phi_1$ and~$\phi_2$ conglomerate the
previous sylvan data.

Why conglomeration?  Suppose that $\alpha$ and~$\beta$ are the source
and target lattice points of a sylvan morphism for~$M_L$.  If $\alpha'
\in \alpha + L$ and $\beta$ are the source and target of another
sylvan morphism for~$M_L$, with target degree $\alpha' \neq \alpha$
but the same source degree~$\beta$, then the relevant $R$-free modules
generated in degrees $\alpha$ and~$\alpha'$ are independent of one
another, but modulo the action of~$L$ they become identified.  For
instance, the two sylvan morphisms (a) and~(b) from $\CC_0$
to~$\CC_{-1}$ in the first group of six have the same source degree
$\beta = 120$ but different target degrees $\alpha = 000$ and $\alpha'
= 12\{-2\}$ that become identified modulo~$L$ because $\alpha' - \alpha
= \alpha' \in L$.  (It is worth checking this in the staircase
diagram.)  Translating this setup by~$L$, the same situation can be
viewed as letting $\alpha$ remain the same while taking $\beta \neq
\beta' \in \beta + L$.

Using variables $u,v,w$ for the Laurent $L$-monomials in~$R[L]$
corresponding to
$x,y,z$ in~$R$, the lattice module $M_L$ is isomorphic, as an
$R[L]$-module, to the quotient
$$
  R[L]\Big/\!
    \Bigl\<
    x^{a_1}y^{a_2}z^{a_3}-x^{b_1}y^{b_2}z^{b_3}u^{a_1-b_1}v^{a_2-b_2}w^{a_3-b_3}
    \,\bigg|\,
    \begin{array}{@{}l}
      (a_1,a_2,a_3)\;,\,\ (b_1,b_2,b_3)\in \NN^3 \text{ and}\!\!
    \\(a_1,a_2,a_3)-(b_1,b_2,b_3)\in L
    \end{array}
    \Bigr\>.
$$

The homological degree~$0$ syzygy module
$$
  \image(\phi_1)
  =
  \Bigl\<
    \frac{w^2}{uv^2}xy^2-z^2,
    \frac{vw}{u^2}xz-\frac{w^2}{uv^2}y^3,
    yz-\frac{vw}{u^2}x^2
  \Bigr\>
$$
can be read from the first group of six sylvan morphisms of the $R$-module~$M_L$.
For instance, we consider the sylvan morphsims from $\beta= 002$ to $\alpha = 000$ and $\alpha^\prime  = \{-1\}\{-1\}2$ which are the translations of the first two morphisms $\CC_{-1} \from \CC_0$ in the list. We pick a generator for $\HH_0 K^{002}$ and it can be either $x-z$ or $y-z$. Here $x,y,z$ are the $0$-faces of the Koszul complex not the variables in the polynomial ring. Two sylvan morphisms are listed here and we look at the images of the $x-z = [1\,\,0\,\,1]^T$ in each morphism\hfill
$$
\CC_{{\!-\!}1} K^{000} \otimes \<1\>
\xleftarrow{
 \times \frac{w^2}{uv^2}
}
\CC_{{\!-\!}1} K^{\{-1\}\{-2\}2} \otimes \<\frac{w^2}{uv^2}\>
\xleftarrow{
  \monomialmatrix
    {\nothing}
    {\scriptstyle
     \begin{array}{@{\,}c@{\ \,}c@{\ \,}c@{\,}}
     x & y & z
     \\
     1 & 1 & 0
     \end{array}
    }
    {\\}
}
\CC_0 K^{002} \otimes \<\frac{w^2}{uv^2}xy^2\>
$$
\[
\nothing\otimes \frac{w^2}{uv^2}xy^2 \reflectbox{$\xmapsto{\hspace{2cm}}$} \nothing\otimes xy^2\cdot \frac{w^2}{uv^2}\reflectbox{$\xmapsto{\hspace{2.5cm}}$} (x-z)\otimes \frac{w^2}{uv^2}xy^2
\]

$$
\CC_{{\!-\!}1} K^{000} \otimes \<1\>
\xleftarrow{
  \monomialmatrix
    {\nothing}
    {\scriptstyle
     \begin{array}{@{\,}c@{\ \,}c@{\ \,}c@{\,}}
     x & y & z
     \\
     0 & 0 & 1
     \end{array}
    }
    {\\}
}
\CC_0 K^{002} \otimes \<z^2\>
$$
\[
-\nothing\otimes z^2 \reflectbox{$\xmapsto{\hspace{2.5cm}}$} (x-z)\otimes z^2
\]
which correspond to the first generator of syzygy module.
Applying the functor~$\pi_L$ from $\ZZ^3$-graded $R[L]$-modules to
$\ZZ^3/L$-graded $R$-modules \cite[Corollary~3.3]{bayer-sturmfels1998}
(see also \cite[Theorem~9.17]{cca}) yields the resolution of
the lattice ideal~$I_L$ constructed in \cite[Example~9.21]{cca}. We write down the resolutions of $M_L$ as $R[L]$-module
$$
0
\leftarrow
R[L]
\xleftarrow{
  \monomialmatrix
    {}
    {
     \begin{array}{@{\,}c@{\ \,}c@{\ \,}c@{\,}}
     \,\, & \,\, & \,\,
     \\
     \frac{w^2}{uv^2}xy^2-z^2 & \frac{vw}{u^2}xz-\frac{w^2}{uv^2}y^3 & yz-\frac{vw}{u^2}x^2
     \end{array}
    }
    {\\}
}
R[L]^3
\xleftarrow{
  \monomialmatrix
   {}
    {
     \begin{array}{@{\ \ \,}r@{\ \,}c@{\ \,}c@{}}
                \,\,        &        \,\,        &\,\,
     \\y & \frac{xvw}{u^2}
     \\x & z
     \\z & \frac{y^2w^2}{uv^2}
     \end{array}
    }
    {\\\\\\}
}
R[L]^2
\leftarrow
0
$$
and resolution of $R/I_L$ as $R$-module
$$
0
\leftarrow
R
\xleftarrow{
  \monomialmatrix
    {}
    {
     \begin{array}{@{\,}c@{\ \,}c@{\ \,}c@{\,}}
     \,\, & \,\, & \,\,
     \\
     xy^2-z^2 & xz-y^3 & yz-x^2
     \end{array}
    }
    {\\}
}
R^3
\xleftarrow{
  \monomialmatrix
   {}
    {
     \begin{array}{@{\ \ \,}r@{\ \,}c@{\ \,}c@{}}
                \,\,        &        \,\,        &\,\,
     \\y & x
     \\x & z
     \\z & y^2
     \end{array}
    }
    {\\\\\\}
}
R^2
\leftarrow
0.
$$
\begin{remark}\label{r:any-equivariant}
The fact that the $L$-equivariant free resolution of~$M_L$ in this
section---and indeed, in Theorem~\ref{t:sylvan}---is sylvan is largely
irrelevant to the main point of the paper: any canonical minimal
$R$-free $\ZZ^n$-graded resolution of the lattice module~$M_L$ is
automatically $L$-equivariant, so the quotient functor of Bayer and
Sturmfels \cite{bayer-sturmfels1998} yields a similarly canonical
$\ZZ^n/L$-graded minimal $R$-free resolution of the lattice
ideal~$I_L$.
\end{remark}


\end{document}